\documentclass{amsart}
\usepackage{amssymb}
\usepackage{scalerel}
\usepackage{comment}
\usepackage{xcolor}
\usepackage[colorlinks=true,linkcolor=blue,allcolors=blue]{hyperref}

\newtheorem{thm}{Theorem}[section]
\newtheorem{prop}[thm]{Proposition}
\newtheorem{lem}[thm]{Lemma}

\providecommand{\keywords}[1]{%
  \small
  \textbf{\textit{Keywords---}} #1
}

\theoremstyle{definition}
\newtheorem{definition}[thm]{Definition}

\theoremstyle{remark}
\newtheorem{remark}[thm]{Remark}
\numberwithin{equation}{section}

\newcommand{\ip}[2]{\left\langle {#1} , {#2} \right\rangle} 
\newcommand{\norm}[1]{\left\lVert #1 \right\rVert} 
\newcommand{\R}{\mathbb{R}} 
\newcommand{\C}{\mathbb{C}} 
\newcommand{\D}{\mathbb{D}} 
\newcommand{\T}{\mathbb{T}} 

\begin{document}

\title[Nonexistence of Dual Frames]{Nonexistence of Approximately Dual Frames via Weighted Composition Operators}

\author{Trevor Camper}
\address{Department of Mathematics, Dartmouth College,
Hanover, NH, USA, 03755}
\email{trevor.camper@dartmouth.edu}

\author{Dongwei Chen}
\address{Department of Mathematics and Statistical Science, University of Idaho, Moscow, ID, USA, 83844}
\email{dongweic@uidaho.edu}

\author{Shuang Guan}
\address{School of Mathematics, Georgia Institute of Technology, Atlanta, GA, USA, 30332}
\email{sguan7@gatech.edu}

\begin{abstract}
    Reconstruction from the coefficients of a frame often requires a dual frame, and the canonical dual is obtained by inverting the frame operator, which is rarely available in a closed form.  It is therefore natural to ask whether the duality can instead be implemented by an operator of prescribed form. We study this problem in the context of weighted composition operators on Hardy and Bargmann-Fock spaces. We first characterize those weighted composition operators which carry a frame together with a dual of it to another such pair.  We then prove that if the image of a frame under a weighted composition operator is itself a dual frame, then such an operator is a positive scalar multiple of the identity and the frame is tight on both Hardy and Bargmann-Fock spaces. Relaxing duality to approximate duality, we show that on the Hardy space, composition is still forced to be the identity under mild regularity for the weight function. Approximate tightness, however, fails completely.  On the Bargmann-Fock space, the situation is rigid: such an approximate dual exists precisely when the ratio of the optimal frame bounds does not exceed an explicit sharp threshold determined by the prescribed error.
\end{abstract}



\maketitle

\section{Introduction}

Initiated by the seminal study of Duffin and Schaeffer \cite{duffin1952class}, frame theory was introduced in the course of studying nonharmonic Fourier series. A sequence $\{ f_n\}$ in a separable Hilbert space $\mathcal{H}$ is a \emph{frame} if there are constants $0< A \leq B <\infty$ such that
\begin{equation*}
    A ||f||^2 \leq \sum_{n=1}^\infty | \langle f ,f_n \rangle |^2 \leq B||f||^2 \quad \text{for every }\; f \in \mathcal{H}.
\end{equation*}
The two inequalities say precisely that the map $f \mapsto \{ \langle f , f_n \rangle\}$ is bounded from above and below, so that $f$ is determined by those coefficients in a numerically stable way. Since then, frames have been applied in many areas, such as the Kadison-Singer problem \cite{casazza2013kadison},  time-frequency analysis \cite{grochenig2001foundations}, wavelet analysis \cite{daubechies1992ten}, compressed sensing~\cite{naidu2020construction}, image processing~\cite{kutyniok2012multiscale},  and the design of filter banks~\cite{fickus2013finite}.  
What distinguishes a frame from an orthonormal basis in $\mathcal{H}$ is that linear independence may not be required: a frame can be redundant. That redundancy is not a defect but the point of the theory, as it ensures the robustness of signal reconstruction to noise.

Given a frame $\{ f_n\}$ in $\mathcal{H}$, the signal reconstruction is determined by the associated frame operator
$S:\mathcal{H}\to\mathcal{H}$ given by $    Sf=\sum\limits_{n=1}^{\infty}\ip{f}{f_{n}}f_{n}, \ f \in \mathcal{H}$,
and its dual frames. The best case is the \emph{tight frame} with frame bound $A>0$: the frame operator $S = A I$ where $I$ is the identity operator in $\mathcal{H}$ and the signal reconstruction is given by $f = \frac{1}{A}Sf = \frac{1}{A}\sum\limits_{n=1}^{\infty}\ip{f}{f_{n}}f_{n}, \ f \in \mathcal{H}$. For the dual frame reconstruction, recall that a sequence  $\{g_{n}\} \subset \mathcal{H}$ is called a dual frame of $\{f_{n}\}$ if for each $h\in \mathcal{H}$,
    \begin{align*}
        h=\sum_{n=1}^{\infty}\ip{h}{f_{n}}g_{n}=\sum_{n=1}^{\infty}\ip{h}{g_{n}}f_{n}.
    \end{align*} 
Note that the frame $\{ f_n\}$ has at least one dual frame $\{ S^{-1} f_n\}$ obtained by computing the inverse of the frame operator $S$, which is called the \emph{canonical dual frame},  but that inverse is computationally inefficient. Christensen and Laugesen \cite{CL10} proposed relaxing the exact reconstruction, asking only that the reconstruction error be at most $\varepsilon$ times the norm of the function. Such a sequence is an \emph{$\varepsilon$-approximate dual frame}, and $\varepsilon$ is larger than the worst-case relative error. Approximate duals are easier to obtain than exact ones and they can be sought inside a prescribed class of operators. 
Frames have been studied using tools from many mathematical areas, including operator theory \cite{christensen2017operator}, combinatorics \cite{bownik2015combinatorial}, matroid theory \cite{bernstein2020algebraic}, probability theory and optimal transport \cite{ehler2013probabilistic}, and algebraic geometry \cite{cahill2013algebraic}. For more background on frame theory, see \cite{christensen2003introduction}.

In this paper, we study frames in the context of complex function spaces.  Specifically, we study frames using the  \emph{weighted composition operator} (WCO):
\begin{equation*}
    W_{\psi,\varphi}f = \psi \cdot (f\circ \varphi),
\end{equation*}
which includes the multiplication and composition operators with composition function $\varphi$ and weight function $\psi$. These operators are efficient to apply, requiring only multiplication and one change of variable rather than a spectral calculation. When such an operator maps frames to frames is well understood in many Hilbert spaces of analytic functions (see \cite{dong2024weighted,dong2024weighted2,manhas2019weighted,manhas2024fredholm}, for example). In particular, a motivating result to the present study is the main result of \cite{manhas2019weighted} concerning the preservation of frames in the general weighted Bergman space $A^{2}_{\alpha}(\mathbb{B}_{n})$: 
\begin{thm}[\cite{manhas2019weighted}]
    Suppose $\alpha>-(n+1)$. Let $\psi$ be a holomorphic function on $\mathbb{B}_{n}$ and $\varphi$ a holomorphic self-map of $\mathbb{B}_{n}$ such that the weighted composition operator $W_{\psi,\varphi}$ is bounded on $A^{2}_{\alpha}(\mathbb{B}_{n})$. The following statements are equivalent:
    \begin{itemize}
        \item[1.] $W_{\psi,\varphi}$ preserves frames on $A^{2}_{\alpha}(\mathbb{B}_{n})$. 
        \item[2.] $W_{\psi,\varphi}$ is bounded below on $A^{2}_{\alpha}(\mathbb{B}_{n})$.
        \item[3.] $W_{\psi,\varphi}$ is surjective on $A^{2}_{\alpha}(\mathbb{B}_{n})$.
        \item[4.] $W_{\psi,\varphi}$ is invertible on $A^{2}_{\alpha}(\mathbb{B}_{n})$.
        \item[5.] $\psi$ and $1/\psi$ are bounded, and $\varphi$ is an automorphism of $\mathbb{B}_{n}$. 
    \end{itemize}
\end{thm}

What has not been understood, to the best knowledge of the authors, is under what circumstances WCOs generate dual or approximate dual frames in a separable Hilbert space $\mathcal{H}$. That is, given a frame $\{f_{n}\}$ for $\mathcal{H}$ and a WCO $W:\mathcal{H}\to\mathcal{H}$, under what conditions  is $\{Wf_{n}\}$ dual to $\{f_{n}\}$? For the present study, we will focus on the settings of the Hardy space $H^{2}(\mathbb{D})$ and the Bargmann-Fock space $\mathcal{F}^{2}(\mathbb{C}^{n})$. 
The main results of this paper will show that in the case of the Hardy and Bargmann-Fock spaces, essentially all WCOs with non-trivial composition operators cannot generate dual frames:
\begin{thm}\label{main_1}
    Let $\mathcal{H}=H^{2}(\mathbb{D})\hspace{.15cm}\text{or}\hspace{.15cm}\mathcal{F}^{2}(\mathbb{C}^{n})$, and let $\{f_{n}\}$ be a frame for $\mathcal{H}$ with frame operator $S$. If $W=W_{\psi,\varphi}$ is a bounded WCO such that $\{Wf_{n}\}$ is a dual frame to $\{f_{n}\}$, then $W = \lambda I$ for some $\lambda>0$, and $\{f_{n}\}$ and $\{Wf_{n}\}$ must be tight. Furthermore, $W = S^{-1}$ and $\{Wf_{n}\}$ is the canonical dual frame of $\{f_{n}\}$.
\end{thm}
The proof of Theorem \ref{main_1} will be presented across Theorem \ref{no_dual} and Proposition \ref{fock_tight}. Surprisingly, this result extends to the case of approximate dual frames in the Hardy space, where we show that under mild assumptions on the symbol of the multiplication operator the only possible WCOs which generate approximate duals are those for which the composition operator is the identity (see Theorem \ref{approx_dual}).  Finally, we consider the implications on tightness of the frame when duals can be generated by WCOs. In this, we obtain strikingly different results between the Hardy and Bargmann-Fock spaces:
\begin{thm}\label{main_2}
    Let $0<\varepsilon<1$. 
    \begin{itemize}
        \item[1.] There are bounded WCOs $W:H^{2}(\mathbb{D})\to H^{2}(\mathbb{D})$ and frames $\{f_{n}\}$ such that $\{Wf_{n}\}$ is an $\varepsilon$-approximate dual frame to $\{f_{n}\}$, but $\{f_{n}\}$ is not approximately tight.
        \item[2.] Let $\{f_{n}\}$ be a frame for $\mathcal{F}^{2}(\mathbb{C}^{n})$ with optimal frame bounds $0<A\leq B<\infty$. Then there exists a bounded WCO $W:\mathcal{F}^{2}(\mathbb{C}^{n})\to\mathcal{F}^{2}(\mathbb{C}^{n})$ such that $\{Wf_{n}\}$ is an $\varepsilon$-approximate dual frame to $\{f_{n}\}$ if and only if 
        \begin{align*}
            \frac{B-A}{A+B}\leq\varepsilon.
        \end{align*}
    \end{itemize}
\end{thm}
The latter result seems to imply that $\varepsilon$-approximate dual frames can be generated from WCOs only in the case of starting with approximately tight frames in the setting of the Bargmann-Fock space, while in the Hardy space this is not guaranteed. The proof of Theorem \ref{main_2} will be presented across Propositions \ref{thm:hardy_approximate_dual} and \ref{fock_main}. 

The paper is structured as follows. In Section \ref{sec:prelim}, we will present some preliminary material and definitions. Section \ref{sec:dualpreservation_Hardy} will prove the Hardy space portions of Theorems \ref{main_1} and \ref{main_2}. Section \ref{sec:dualpreservation_Hardy} will also include results concerning approximate duals in the Hardy space. Section \ref{sec:dualpreservation_Fock} will cover the Bargmann-Fock space portions of Theorems \ref{main_1} and \ref{main_2} and conclude the paper.

\section{Preliminaries}\label{sec:prelim}
Before proceeding to the main results of this paper, we first cover some preliminary material on WCOs and frames. Standard references for this material are \cite{christensen2003introduction,cowen2019composition,shapiro2012composition}. First, we will cover the relevant spaces of analytic functions in which we will consider frames. Throughout the paper, $\mathbb{D}$ refers to the unit disc in the complex plane $\mathbb{C}$ without boundary, $\mathbb{T}$ is the unit circle in $\mathbb{C}$, and $\mathcal{O}(\Omega)$ refers to the set of holomorphic functions over a domain $\Omega\subset\mathbb{C}^{n}$. Perhaps the most classical space of holomorphic functions is the Hardy space.
\begin{definition}\label{def:Hardy}
    Let $1\leq p<\infty$. We call the set of functions $f\in\mathcal{O}(\mathbb{D})$ such that 
    \begin{align*}
        \sup_{0\leq r<1}\int_{0}^{2\pi}\left|f(re^{i\theta})\right|^{p}d\theta<\infty
    \end{align*}
    the $p$-Hardy space over the unit disc $\mathbb{D}$ and denote it by $H^{p}(\mathbb{D})$. If $p=\infty$, we denote the set of functions $f\in\mathcal{O}(\mathbb{D})$ such that 
    \begin{align*}
        \sup_{|z|<1}|f(z)|<\infty
    \end{align*}
    as the space $H^{\infty}(\mathbb{D})$. 
\end{definition}
It is well-known \cite{duren1970theory,paulsen2016introduction} that $H^{2}(\mathbb{D})$ is a Reproducing Kernel Hilbert Space (RKHS), i.e., for each $w\in\mathbb{D}$ the point evaluation functional $\delta_{w}:H^{2}(\mathbb{D})\to\mathbb{C}$ given by $\delta_w(f) = f(w)$ is bounded. An application of the Riesz Representation Theorem \cite{lax2014functional} then gives the existence of vectors $K_{w}\in H^{2}(\mathbb{D})$, called \textit{reproducing kernels} at $w$ such that $\ip{f}{K_{w}}=f(w)$. In the case of the Hardy space, the reproducing kernels are 
\begin{align*}
    K_{w}(z)=\frac{1}{1-\overline{w}z},
\end{align*}
and the inner product is induced by the integral above. The other analytic function space of interest for this paper will be the Bargmann-Fock space \cite{Z12}.
\begin{definition}
Let $dV$ denote the Lebesgue volume measure on $\mathbb{C}^{n}$ and set
\[
  d\lambda(z)=\pi^{-n}e^{-\|z\|^{2}}\,dV(z).
\]
We call the set of functions $f\in\mathcal{O}(\mathbb{C}^{n})$ such that
\[
  \|f\|^{2}=\int_{\mathbb{C}^{n}}|f(z)|^{2}\,d\lambda(z)<\infty
\]
the \emph{Bargmann-Fock space}, and denote it by $\mathcal{F}^{2}(\mathbb{C}^{n})$. Furthermore, $\mathcal{F}^{2}(\mathbb{C}^{n})$ is a Hilbert space, and
the inner product is given as
$$\langle f,g\rangle=\int_{\mathbb{C}^{n}}f(z)\overline{g}(z)\,d\lambda(z).$$
\end{definition}
Again, it is known \cite{Z12} that the Bargmann-Fock space $\mathcal{F}^{2}(\mathbb{C}^{n})$ is an RKHS with reproducing kernel 
\begin{align*}
    K_{w}(z)=e^{\ip{z}{w}}. 
\end{align*}
For the entirety of this paper, we will be interested in frames on these spaces. 
\begin{definition}
    Let $\mathcal{H}$ be a separable Hilbert space. We call a sequence of vectors $\{f_{n}\} \subset \mathcal{H}$ a frame if there are constants $0<A\leq B<\infty$ such that for each $f\in\mathcal{H}$,
    \begin{align*}
        A\norm{f}^{2}\leq \sum_{n=1}^{\infty}\left|\ip{f}{f_{n}}\right|^{2}\leq B\norm{f}^{2}.
    \end{align*}
    We say the frame is tight if $A=B$ and Parseval if $A=B=1$. In particular, $\{f_{n}\}\subset\mathcal{H}$ is called a Bessel sequence if only the upper bound holds.  
\end{definition}
While many such constants $A$ and $B$ may satisfy the above frame inequality, we will assume $A$ and $B$ are the \textit{optimal} constants, i.e. $A$ is the largest such lower bound while $B$ is the smallest such upper bound. Associated to a given frame $\{f_{n}\} \subset \mathcal{H}$ is an operator $S:\mathcal{H}\to\mathcal{H}$ by 
\begin{align*}
    Sf=\sum_{n=1}^{\infty}\ip{f}{f_{n}}f_{n}, \ f \in \mathcal{H}.
\end{align*}
This operator is referred to as the \textit{frame operator} associated to $\{f_{n}\}$ and will be central to the results to follow. Below, we collect several well-known and standard facts about the frame operator without proof \cite{christensen2003introduction}, which will be used throughout. 
\begin{lem}\label{lemma:frame_operator}
    Let $\mathcal{H}$ be a separable Hilbert space and  $\{f_{n}\}$ be a frame for $\mathcal{H}$ with bounds $0<A\leq B<\infty$. If $S:\mathcal{H}\to\mathcal{H}$ is the associated frame operator, then 
    \begin{itemize}
        \item[1.] $S$ is bounded, positive, and invertible with a bounded inverse. 
        \item[2.] $\sigma(S)\subset [A,B]$, where $\sigma(S)$ is the spectrum of $S$. 
    \end{itemize}
\end{lem}

Often, the ideal circumstance is when $\{f_{n}\}$ is a Parseval frame as the frame operator is the identity and we recover a Parseval-like identity. However, this is not always the case. In this case, one can hope that by finding a second family of vectors one can recover a Parseval-type identity to reconstruct signals. This motivates the following definition of dual frames. 

\begin{definition}
    Let $\{f_{n}\}$ be a frame for a separable Hilbert space $\mathcal{H}$. A sequence of vectors $\{g_{n}\}\subset \mathcal{H}$ is said to be a dual frame of  $\{f_{n}\}$ if for each $h\in \mathcal{H}$,
    \begin{align*}
        h=\sum_{n=1}^{\infty}\ip{h}{f_{n}}g_{n}=\sum_{n=1}^{\infty}\ip{h}{g_{n}}f_{n}.
    \end{align*}
    In particular, we call the operator $S_{f,g}:\mathcal{H}\to \mathcal{H}$ by 
    \begin{align*}
        S_{f,g}h=\sum_{n=1}^{\infty}\ip{h}{f_{n}}g_{n}, \ h \in \mathcal{H},
    \end{align*}
    the mixed frame operator associated to the dual pair $\{f_{n}\}$ and $\{g_{n}\}$. 
\end{definition}
For a given frame $\{f_{n}\}$ with frame operator $S$, dual frames are not necessarily unique. However, it is known that one can always construct a so-called canonical dual frame $\{S^{-1}f_{n}\}$ using the frame operator. Indeed, every dual frame $\{g_n\}$ of $\{f_n\}$ can be constructed from the canonical dual frame $\{S^{-1}f_{n}\}$ via the formula
\begin{equation*}
    g_{n}=S^{-1}f_{n}+h_{n}, \ \text{for each $n$},
\end{equation*}
where $\{h_{n}\}$ is a Bessel sequence in $\mathcal{H}$ such that for each $f\in\mathcal{H}$,
\begin{align*}
    \sum_{n=1}^{\infty}\ip{f}{f_{n}}h_{n}=0.
\end{align*}
See \cite{christensen2003introduction} for more details on the characteristics of dual frames. 
However, it can be relatively difficult to compute the inverse $S^{-1}$ depending on the circumstances. Often when one does not have access to a full dual frame, the best hope is to obtain the Parseval condition in an approximate sense, which leads to the notion of approximately dual frames. 
\begin{definition}[\cite{CL10}]\label{def_approx_dual}
    Let $0\leq\varepsilon<1$. Suppose $\mathcal{H}$ is a separable Hilbert space and $\{f_{n}\}$ is a frame for $\mathcal{H}$. We say that a sequence of vectors $\{g_{n}\} \subset \mathcal{H}$ is an $\varepsilon$-approximate dual frame to $\{f_{n}\}$ if the norm of the mixed frame operator $S_{f,g}$ for $\{f_n\}$ and $\{g_n\}$ satisfies 
    $\norm{I-S_{f,g}}\leq\varepsilon$. 
\end{definition}
Finally, we define our main object of study: the weighted composition operator. 

\begin{definition}
Let $\Omega$ be either $\mathbb{D}$ or $\mathbb{C}^{n}$ and $\mathcal{H}$ be either $H^2(\mathbb{D})$ or $\mathcal{F}^2(\mathbb{C}^n)$. Suppose $\psi:\Omega\to\mathbb{C}$ and $\varphi:\Omega\to\Omega$ are functions, and let $M_{\psi}$ and $C_{\varphi}$ denote the operators which multiply by $\psi$ and compose with $\varphi$, respectively. We call the operator $W_{\psi,\varphi}: \mathcal{H} \to \mathcal{H}$ defined by
\begin{align*}
    W_{\psi,\varphi}f(z)=M_{\psi}C_{\varphi}f(z)=\psi(z)f\circ\varphi(z) = \psi(z)f(\varphi(z)), \ f \in \mathcal{H},
\end{align*}
the \emph{Weighted Composition Operator} (WCO) with weight $\psi$ and composition $\varphi$. When not ambiguous, we write $W=W_{\psi,\varphi}$. 
\end{definition}
Finally, an important fact we will use in the results is the following statement about the invertibility of WCOs, established by Gunatillake \cite{gunatillake2011invertible}.
\begin{thm}\cite[Theorem 2.0.1]{gunatillake2011invertible}\label{thm:Gunatillake} The weighted composition operator $W_{\psi,\varphi}$ on $H^2(\D)$ with weight $\psi$ and composition $\varphi$ is invertible if and only if $\psi$ is both bounded and bounded away from zero on the unit disc and $\varphi$ is an automorphism of the unit disc. The inverse operator is the weighted composition operator $W_{1/\psi \circ \varphi^{-1},\varphi^{-1}}.$
    
\end{thm}

\section{Duality Preservation on \texorpdfstring{$H^{2}(\mathbb{D})$}{H²(D)}}\label{sec:dualpreservation_Hardy}

In this section, we study the preservation and generation of dual frames in the Hardy space $H^{2}(\mathbb{D})$. We begin with a result concerning duality preservation of dual pairs of frames. Essentially, the first motivating question is the following: for a frame $\{f_{n}\}$ of $H^{2}(\mathbb{D})$ with dual frame $\{g_{n}\}$, are there conditions on $W=W_{\psi,\varphi}$ which preserves duality, i.e., $\{Wf_{n}\}$ and $\{Wg_{n}\}$ are not only frames for $H^{2}(\mathbb{D})$ but also $\{Wg_{n}\}$ is dual to $\{Wf_{n}\}$? We answer this question by the following proposition. 
\begin{prop}\label{W_on_both}
    Let $\{f_{n}\}$ be a frame for $H^{2}(\mathbb{D})$ and let $\{g_{n}\}$ be a dual frame of $\{f_{n}\}$. Suppose $W:H^{2}(\mathbb{D})\to H^{2}(\mathbb{D})$ is a bounded weighted composition operator. Then $\{Wf_{n}\}$ is a frame with dual frame $\{Wg_{n}\}$ if and only if $W=W_{\psi,\varphi}$ satisfies 
    \begin{equation*}
        \varphi(z)=e^{i\theta}\frac{z-a}{1-\overline{a}z}, 
    \end{equation*}
    and 
    \begin{equation*}
        \psi(z) = c \sqrt{\varphi'(z)}
    \end{equation*}
    for some $a\in\mathbb{D}$ and $\theta, c \in\mathbb{T}$.
    Here $\psi$ is a branch of holomorphic square root of $\varphi'$ on $\D$.
\end{prop}
\begin{proof}
Suppose first that $\{ Wg_n \}$ is a dual frame of $\{ W f_n\}$. For every $h \in H^2(\D)$, it holds that
\begin{equation*}
    h = \sum_n \langle h,Wf_n\rangle Wg_n =W\Big( \sum_n \langle W^*h, f_n \rangle g_n \Big) =WW^*h,
\end{equation*}
the series converges because $\{ f_n\}$ and $\{g_n\}$ are dual; hence $WW^* = I$. Next, for $w \in \D$ and $f \in H^2(\D)$, denoting $K_w$ as the reproducing kernel at $w$ defined after Definition~\ref{def:Hardy} gives
\begin{equation*}
    \langle f, W^*K_w \rangle = \langle Wf, K_w \rangle =\psi(w)f(\varphi(w)) =\langle f, \overline{\psi(w)}K_{\varphi(w)}\rangle.
\end{equation*}
Since $f \in H^2(\D)$ is arbitrary, it holds that $W^*K_w = \overline{\psi(w)}K_{\varphi(w)}$. Now fix $w \in \D$, it holds that for every $z \in \D$,
\begin{equation*}
    (WW^* K_w)(z) = \overline{{\psi(w)}}(WK_{\varphi(w)})(z) = \overline{\psi(w)}\psi(z) K_{\varphi(w)}(\varphi(z)).
\end{equation*}
On the other hand, since $WW^* =I$, $WW^*K_w = K_w$ is an element of $H^2(\D)$ and hence acts as a holomorphic function on $\D$. Therefore,
$$K_w(z) = \overline{\psi(w)}\psi(z) K_{\varphi(w)}(\varphi(z)), \qquad z,w \in \D$$
Writing down the reproducing kernel explicitly, it holds that:
\begin{equation}\label{eq:3.1.1}
    \frac{1}{1-\overline{w}z} = \frac{\overline{\psi(w)}\psi(z)}{1- \overline{\varphi(w)}\varphi(z)}, \quad z,w\in \D.
\end{equation}
Notice that both sides of the denominator are nonzero since $|\overline{w}z|<1$ and $\varphi$ maps $\D$ to $\D$. Taking $w =z$ in Equation~\eqref{eq:3.1.1},
\begin{equation}\label{eq:3.1.2}
    |\psi(z)|^2 (1-|z|^2) = 1 -|\varphi(z)|^2, \quad z \in \D.
\end{equation}
Now take moduli in \eqref{eq:3.1.1} and substitute \eqref{eq:3.1.2} at $z$ and at
$w$. After rearranging,
\[
  \frac{\bigl(1-|\varphi(z)|^{2}\bigr)\bigl(1-|\varphi(w)|^{2}\bigr)}
       {\bigl|1-\overline{\varphi(w)}\varphi(z)\bigr|^{2}}
  =\frac{\bigl(1-|z|^{2}\bigr)\bigl(1-|w|^{2}\bigr)}
        {\bigl|1-\overline{w}z\bigr|^{2}} .
\]
By the equality case
of the Schwarz-Pick lemma \cite{ahlfors1979complex}, $\varphi$ is an automorphism of $\D$. By the classical description of $\text{Aut}(\D)$, every such map is a M\"obius transformation
\[
  \varphi(z)=e^{i\theta}\frac{z-a}{1-\overline{a}z}
  \qquad\text{with }a\in\D,\ \theta\in\R ,
\]
where $a=\varphi^{-1}(0)$ is the unique zero of $\varphi$ in $\D$. For such a
$\varphi$ a direct computation gives
$1-|\varphi(z)|^{2}=\bigl(1-|z|^{2}\bigr)|\varphi'(z)|$, so \eqref{eq:3.1.2}
becomes $|\psi|^{2}=|\varphi'|$ on $\D$.  As $\varphi'$ is zero-free, the
quotient $\psi^{2}/\varphi'$ is holomorphic on $\D$ with constant modulus $1$;
a nonconstant holomorphic function is an open map and the unit circle has empty
interior, so $\psi^{2}/\varphi'$ is a unimodular constant and
$\psi=c\sqrt{\varphi'}$ with $c\in\T$ with a choice of branch cut. 

Conversely, assume that $\varphi$ and $\psi$ have the stated forms. A direct
computation about $\varphi$ and $\psi$ gives
\[
    \frac{\overline{\psi(w)}\psi(z)}
    {1-\overline{\varphi(w)}\varphi(z)}
    =\frac{1}{1-\overline{w}z},
    \qquad z,w\in\D.
\]
Since $W^*K_w=\overline{\psi(w)}K_{\varphi(w)}$, it follows that
\[
    (WW^*K_w)(z)
    =\frac{\overline{\psi(w)}\psi(z)}
    {1-\overline{\varphi(w)}\varphi(z)}
    =K_w(z).
\]
As the linear span of $\{K_w:w\in\D\}$ is dense in $H^2(\D)$, we obtain
$WW^*=I$. Consequently, $W^*$ is an isometry.
 
 An isometry is bounded below, so
its range is closed.  That range contains $\overline{\psi(w)}K_{\varphi(w)}$ for
every $w\in\D$, and $\psi=c\sqrt{\varphi'}$ is zero-free while $\varphi$ maps
$\D$ onto $\D$; hence the range contains every reproducing kernel, so it is
dense, and being closed it is all of $H^{2}(\D)$.  Thus $W^{*}$ is a surjective
isometry, hence unitary, and so is $W=(W^{*})^{*}$.  Consequently
\[
  \sum_{n}\bigl|\ip{h}{Wf_n}\bigr|^{2}
  =\sum_{n}\bigl|\ip{W^{*}h}{f_n}\bigr|^{2}
\]
lies between $A\|W^{*}h\|^{2}=A\|h\|^{2}$ and $B\|h\|^{2}$, so $\{Wf_n\}$ is a
frame with the same bounds as $\{f_n\}$; and
$\sum_{n}\ip{h}{Wf_n}Wg_n=WW^{*}h=h$, with the same identity after exchanging
the roles of $\{f_n\}$ and $\{g_n\}$, so $\{Wg_n\}$ is a dual frame of
$\{Wf_n\}$.
\end{proof}

The previous result about the preservation of dual frames in some sense can be viewed as a companion result to the preservation of frames obtained in \cite{manhas2019weighted}. As discussed in the introduction, there are canonical ways of generating both dual and approximate dual frames for a given frame $\{f_{n}\}$. In particular, the use of the frame operator $S$ of frame $\{f_{n}\}$ serves as the primary example to generate the canonical dual frame. One could hope that applying an operation much simpler than the frame operator, such as a WCO, would also produce a dual frame. However, this is not the case in the Hardy space $H^{2}(\mathbb{D})$ as we will see in Theorem \ref{no_dual}, which claims that given a frame $\{f_{n}\}$ in $H^{2}(\mathbb{D})$ with frame operator $S$, if $\{Wf_{n}\}$ is a dual frame for $\{f_{n}\}$ where $W:H^{2}(\mathbb{D})\to H^{2}(\mathbb{D})$ is a bounded WCO,  then $\{f_{n}\}$ and $\{Wf_{n}\}$ must be tight frames, $W = S^{-1}$, and $\{Wf_{n}\}$ is the canonical dual frame of $\{f_{n}\}$.

\begin{thm}\label{no_dual}
    Let $\{f_{n}\}$ be a frame for $H^{2}(\mathbb{D})$ with frame operator $S$ such that $\{Wf_{n}\}$ is a dual frame of $\{f_{n}\}$, where $W:H^{2}(\mathbb{D})\to H^{2}(\mathbb{D})$ is a bounded weighted composition operator. Then $W= \lambda I$ for some $\lambda>0$ and $\{f_{n}\}$ must be tight. Furthermore, $W=S^{-1}$, $\{ Wf_n\}$ is the canonical dual of $\{f_{n}\}$, and $\{ Wf_n\}$ is tight.
\end{thm}
\begin{proof}
For every $f \in H^2(\D)$, boundedness of $W$ gives 
\begin{equation*}
    f = \sum_n \langle f , f_n \rangle Wf_n = W \Big( \sum_n \langle f,f_n\rangle f_n\Big) = WSf,
\end{equation*}
    therefore $WS = I$, where $S$ is the frame operator. Following from Lemma~\ref{lemma:frame_operator}, $S$ is positive invertible, $W= S^{-1}$, $W \geq 0$ and $W = W^*$. For $w\in\D$ and $f\in H^{2}(\D)$,
\[
  \ip{f}{W^{*}K_w}=\ip{Wf}{K_w}=(Wf)(w)=\psi(w)f(\varphi(w))
   =\bigl\langle f,\ \overline{\psi(w)}K_{\varphi(w)}\bigr\rangle ,
\]
so $W^{*}K_w=\overline{\psi(w)}K_{\varphi(w)}$.  Comparing this with
$WK_w(z)=\psi(z)K_w(\varphi(z))$ and using $W=W^{*}$ gives
$\psi(z)K_w(\varphi(z))=\overline{\psi(w)}K_{\varphi(w)}(z)$, that is,
\begin{equation}\label{eq:3.2.1}
  \psi(z)\bigl(1-\overline{\varphi(w)}z\bigr)
  =\overline{\psi(w)}\bigl(1-\overline{w}\varphi(z)\bigr),
  \qquad z,w\in\D .
\end{equation}
Setting $w=0$ in \eqref{eq:3.2.1} and
using $K_0\equiv1$ gives $\psi(z)(1-\overline{\varphi(0)}z)=\overline{\psi(0)}$, and
then $z=0$ gives $\psi(0)=\overline{\psi(0)}$.  Thus $\psi(0)$ is real.

We claim that $\psi(0) >0$. In fact, since $W$ is invertible, $\psi\in H^{\infty}(\mathbb{D})\subset H^{2}(\mathbb{D})$. Applying $W$ to the constant function $1$ yields
\begin{equation*}
    0\leq \ip{W1}{1} =\ip{\psi}{1}=  \ip{\psi}{K_0} = \psi(0),  
\end{equation*}
where the first inequality follows since $W\geq 0$. Notice that the invertibility of $W$ (by Theorem~\ref{thm:Gunatillake}) implies that $|\psi|$ is bounded away from zero in $\D$.  Hence,
\begin{equation}\label{eq:3.2.2}
  \psi(z)=\frac{\psi(0)}{1-\overline{\varphi(0)}z},\qquad z\in\D .
\end{equation}
Substituting \eqref{eq:3.2.2} into \eqref{eq:3.2.1}, using
$\overline{\psi(w)}=\psi(0)\bigl(1-\varphi(0)\overline{w}\bigr)^{-1}$ and cancelling
$\psi(0)\neq0$, we obtain
\begin{equation}\label{eq:3.2.3}
  \bigl(1-\varphi(0)\overline{w}\bigr)\bigl(1-\overline{\varphi(w)}z\bigr)
  =\bigl(1-\overline{\varphi(0)}z\bigr)\bigl(1-\overline{w}\varphi(z)\bigr),
  \qquad z,w\in\D .
\end{equation}
Write $\varphi(z)=\sum_{k\ge0}a_kz^{k}$, so that $a_0=\varphi(0)$ and
$\overline{\varphi(w)}=\sum_{k\ge0}\overline{a_k}\,\overline{w}^{\,k}$, and
compare coefficients of $\overline{w}z^{m}$ on the two sides of
\eqref{eq:3.2.3}.  For $m=1$ this gives $a_1=\overline{a_1}$, so $a_1\in\R$; for
$m\ge2$ it gives $a_{m}=\overline{\varphi(0)}\,a_{m-1}$.  Hence
\begin{equation}\label{eq:3.2.4}
  \varphi(z)=\varphi(0)+\frac{a_1z}{1-\overline{\varphi(0)}z}
           =\frac{\varphi(0)+\beta z}{1-\overline{\varphi(0)}z},
  \qquad \beta:=a_1-|\varphi(0)|^{2}\in\R .
\end{equation}
 Since $W$ is invertible, $\varphi$ is an automorphism of $\D$ by Theorem~\ref{thm:Gunatillake}, so
$\varphi(z)=e^{i\theta}(a-z)(1-\overline{a}z)^{-1}$ for some $a\in\D$ and
$\theta\in\R$. The numerator and denominator in
\eqref{eq:3.2.4} have no common zero, since a common zero would force
$\beta=-|\varphi(0)|^{2}$ and hence $\varphi\equiv \varphi(0)$, which is not an automorphism.
The two representations therefore agree up to a scalar, so
$\overline{a}=\overline{\varphi(0)}$ and $\varphi(0)=e^{i\theta}a$.  If $\varphi(0)\neq0$
this forces $e^{i\theta}=1$, so
\[
  \varphi(z)=\frac{\varphi(0)-z}{1-\overline{\varphi(0)}z};
\]
if $\varphi(0)=0$ then \eqref{eq:3.2.4} gives $\varphi(z)=\beta z$ with $\beta$ real and,
$\varphi$ being an automorphism, $|\beta|=1$, so $\varphi(z)=\pm z$.  In every
case $\varphi\circ\varphi=I$. Consequently
$$W^{2}f(z)=\psi(z)\psi(\varphi(z))f(\varphi(\varphi(z)))
 =\psi(z)\psi(\varphi(z))f(z),$$
 with \eqref{eq:3.2.2},
\[
  \psi(z)\psi(\varphi(z))
  =\frac{\psi(0)}{1-\overline{\varphi(0)}z}\cdot
   \frac{\psi(0)\,(1-\overline{\varphi(0)}z)}{1-|\varphi(0)|^{2}}
  =\frac{\psi(0)^{2}}{1-|\varphi(0)|^{2}}=:\lambda^{2},
\]
with $\lambda:=\psi(0)\bigl(1-|\varphi(0)|^{2}\bigr)^{-1/2}>0$.  Thus $W^{2}=\lambda^{2}I$.
Since $W\ge0$, $W$ is a positive square root of $\lambda^{2}I$, and positive
square roots are unique, so $W=\lambda I$. Since $W = S^{-1}$, then $S = \frac{1}{\lambda}I$, which implies $\{f_n\}$ is a tight frame. Applying this to $f\equiv1$ gives
$\psi\equiv\lambda$, and to $f(z)=z$ gives $\lambda\varphi(z)=\lambda z$, so
$\varphi(z)=z$. Since $W = S^{-1}=\lambda I$, $\{ Wf_n\}$ is the canonical dual frame $\{ S^{-1}f_n\}$ of the frame $\{f_n\}$. Since $\{f_n\}$ is a tight frame with frame operator $S = \frac{1}{\lambda}I$, its canonical dual frame $\{ S^{-1}f_n\}$ is also a tight frame with frame operator $\lambda I$.
\end{proof}
\begin{remark}\label{remark:canonical}
Dual frames are far from unique if the given frame is redundant. The situation changes completely, however, once the dual is required to be the image of the frame under a single bounded operator.  Indeed, if $T$ is a bounded operator on $\mathcal H$ such that $\{Tf_n\}$ is a dual frame of the frame $\{f_n\}$, then
for every $f\in\mathcal H$,
\[
  f=\sum_{n}\ip{f}{f_n}Tf_n=T\Bigl(\sum_{n}\ip{f}{f_n}f_n\Bigr)=TSf ,
\]
so $TS=I$ and hence $T=S^{-1}$. In particular $\{Tf_n\}=\{S^{-1}f_n\}$ is the
canonical dual frame. Hence, Theorem~\ref{no_dual} shows that if $\{ Wf_n\}$ is a dual frame of $\{ f_n\}$, then $\{ f_n\}$ must be tight and $\{ Wf_n\}$ must be the canonical dual frame.
\end{remark}

While tightness is forced by the WCO duality in Theorem \ref{no_dual}, one could hope that relaxing duality would result in obtaining a class of WCOs for which $\{Wf_{n}\}$ is an $\varepsilon$-approximate dual frame to $\{f_{n}\}$. Unfortunately, as we will see with our next main result in Theorem \ref{approx_dual}, this is essentially impossible without forcing the WCO to be a multiplication operator. First, we require the following lemma on linear operators, whose proof we include for the sake of completeness. 
 
\begin{lem}\label{invert_lem}
    Let $\mathcal{H}$ be a Hilbert space. Suppose $T:\mathcal{H}\to\mathcal{H}$ is a bounded operator such that $\norm{I-T}<1$. Then 
    \begin{itemize}
        \item[1.] There is a $c>0$ such that $\text{Re} \ip{Tx}{x} \geq c\norm{x}^{2}$ for each $x\in\mathcal{H}$. 
        \item[2.] There is a constant $C>0$ such that $T+T^{*}\geq CI$, in the sense that for each $x\in\mathcal{H}$, $\ip{(T+T^{*})x}{x} \geq C\norm{x}^{2}$.
    \end{itemize}
\end{lem}
\begin{proof}
    Let $k=\norm{I-T}$. For each $x\in\mathcal{H}$, 
    \begin{align*}
        \norm{(I-T)x}\leq k\norm{x},
    \end{align*}
    and squaring this gives $\norm{(I-T)x}^{2}\leq k^{2}\norm{x}^{2}$. Expanding $\norm{(I-T)x}^{2}$, we have 
    \begin{align*}
        \norm{(I-T)x}^{2}=\norm{x}^{2}-2\text{Re}\ip{Tx}{x}+\norm{Tx}^{2}.
    \end{align*}
  Hence 
    \begin{align*}
        \norm{x}^{2}-2\text{Re}\ip{Tx}{x}+\norm{Tx}^{2}\leq k^{2}\norm{x}^{2},
    \end{align*}
    and 
    \begin{align*}
        (1-k^{2})\norm{x}^{2}\leq(1-k^{2})\norm{x}^{2}+\norm{Tx}^{2}\leq 2\text{Re}\ip{Tx}{x}.
    \end{align*}
    Letting $c:=\frac{1-k^{2}}{2}>0$, we have that $\text{Re}\ip{Tx}{x}\geq c\norm{x}^{2}$ for each $x\in\mathcal{H}$.
    For the second part, note that 
    \begin{align*}
        \text{Re}\ip{Tx}{x}=\frac{\ip{Tx}{x}+\overline{\ip{Tx}{x}}}{2}=\frac{\ip{(T+T^{*})x}{x}}{2}. 
    \end{align*}
   Then for each $x\in\mathcal{H}$, 
   $$\ip{(T+T^{*})x}{x} = 2\text{Re}\ip{Tx}{x}\geq  C\norm{x}^2,$$
   where $C=1-k^{2}>0$.
\end{proof}

Note that $\{Wf_{n}\}$ is an $\varepsilon$-approximate dual frame to the frame $\{f_{n}\}$ with frame operator $S$ if and only if $\norm{I-WS}\leq \varepsilon$. Indeed, for any $f \in H^{2}(\mathbb{D})$,
\begin{equation*}
    \|f -   S_{f_n,Wf_n}f  \| = \|f -  \sum_{n=1}^{\infty}\ip{f}{f_{n}}Wf_{n}\| =  \|f -  WSf\|.
\end{equation*}
Using this in the context of Lemma \ref{invert_lem} gives useful restrictions on the spectrum of such $W$, which leads to our next lemma. 

\begin{lem}\label{spec_lem}
    Let $0 \leq \varepsilon<1$, and suppose $\{f_{n}\}$ is a frame for $H^{2}(\mathbb{D})$ with frame bounds $0<A\leq B<\infty$ and frame operator $S$. Let $W=W_{\psi,\varphi}:H^{2}(\mathbb{D})\to H^{2}(\mathbb{D})$ be a bounded WCO such that $\{Wf_{n}\}$ is an $\varepsilon$-approximate dual frame to $\{f_{n}\}$. Then 
    \begin{itemize}
        \item[1.] $W$ is invertible
        \item[2.] $\psi\in H^{\infty}(\mathbb{D})$,
        \item[3.] $1/\psi\in H^{\infty}(\mathbb{D})$,
        \item[4.] $\varphi(z)$ is a disc automorphism,
        \item[5.] there is a constant $c>0$ such that $\sigma(W)\subset\{\lambda\in\mathbb{C}:\text{Re} \lambda>c\}$.
    \end{itemize}
\end{lem}
\begin{proof}
    Since $\{Wf_{n}\}$ is an $\varepsilon$-approximate dual, we have that $\norm{I-WS}\leq\varepsilon<1$. Hence $WS$ is invertible, and by the invertibility of $S$, $W$ is invertible. The first four points of the lemma then follow by an application of Theorem \ref{thm:Gunatillake}. For the last point, by letting $T=WS$ in the second part of Lemma \ref{invert_lem}, there is a constant $C>0$ such that 
    \begin{align*}
        T+T^{*}\geq CI.
    \end{align*}
    Recalling that $S$ is positive and invertible, by conjugating this expression by $S^{-1/2}$, we obtain 
    \begin{align*}
        S^{-1/2}\left(T+T^{*}\right)S^{-1/2}\geq CS^{-1}\geq (C/B)I.
    \end{align*}
    Using the first part of Lemma \ref{invert_lem} on the operator $S^{-1/2}TS^{-1/2}$, we have that 
    \begin{align*}
        \text{Re}\ip{S^{-1/2}TS^{-1/2}f}{f}\geq (C/B)\norm{f}_{H^{2}}^{2}.
    \end{align*}
    By the Spectral Inclusion Theorem \cite{lax2014functional}, the spectrum of an operator is contained in the closure of its numerical range. This implies that $\sigma(S^{-1/2}TS^{-1/2})\subset\{\lambda \in \mathbb{C}: \text{Re}\lambda\geq (C/B)\}$. However, note that 
    \begin{align*}
        S^{-1/2}TS^{-1/2}=S^{-1/2}WSS^{-1/2}=S^{-1/2}WS^{1/2}.
    \end{align*}
    Hence, $W$ is similar to $S^{-1/2}TS^{-1/2}$ and $\sigma(W)\subset\{\lambda \in \mathbb{C}: \text{Re}\lambda\geq (C/B)\}$. Taking any $0<c<C/B$ gives the claimed conclusion.
\end{proof}
Before proceeding onto the main result of this section, we require the following definition which classifies the class of disc automorphisms into three separate classes. 
\begin{definition}[\cite{gunatillake2011invertible}]
    Let $\varphi:\mathbb{D}\to\mathbb{D}$ be a disc automorphism. We say that 
    \begin{itemize}
        \item[1.] $\varphi$ is \textit{elliptic} if $\varphi$ has a fixed point in $\mathbb{D}$;
        \item[2.] $\varphi$ is \textit{parabolic} if $\varphi$ has only one fixed point on $\mathbb{T}$;
        \item[3.] $\varphi$ is \textit{hyperbolic} if $\varphi$ has two distinct fixed points on $\mathbb{T}$. 
    \end{itemize}
\end{definition}
In the case of hyperbolic automorphisms, the fixed point $\xi$ for which $\varphi'(\xi)<1$ is called the \textit{Denjoy-Wolff point}. 
The spectral theory of invertible WCOs on the Hardy space has been well characterized \cite{gunatillake2011invertible}. As these results will become central to our next main result, we also restate these results here. 

\begin{thm}[Theorems 3.1.1, 3.2.1, 3.3.1, and 3.5.1 of \cite{gunatillake2011invertible}]\label{spectrum_WCO}
    Let $W_{\psi,\varphi}:H^{2}(\mathbb{D})\to H^{2}(\mathbb{D})$ be a bounded weighted composition operator on $H^{2}(\mathbb{D})$. 
    \begin{itemize}
        \item[1.] Suppose $\varphi:\mathbb{D}\to\mathbb{D}$ is an elliptic automorphism with fixed point $a$ such that $\lambda:=\varphi'(a)=e^{2\pi i/n}$ for some $n\in\mathbb{N}$, and let $\phi(z)=(a-z)/(1-\overline{a}z)$. Then the spectrum of $W_{\psi,\varphi}$ is the closure of the set 
        \begin{align*}
            \{\mu:\mu^{n}=\psi\circ\phi(\xi)\psi\circ\phi(\lambda\xi)\cdots\psi\circ\phi(\lambda^{n-1}\xi),\xi\in\mathbb{D}\}.
        \end{align*}
        \item[2.] Suppose $\varphi:\mathbb{D}\to\mathbb{D}$ is an elliptic automorphism with fixed point $a$ such that $\varphi'(a)=e^{2\pi i\theta}$ with $\theta$ irrational and $\psi:\mathbb{D}\to\mathbb{C}$ is continuous on $\overline{\mathbb{D}}$ and bounded below. Then the spectrum of $W_{\psi,\varphi}$ is the circle centered at the origin of radius $\psi(a)$. 
        \item[3.] Suppose $\varphi:\mathbb{D}\to\mathbb{D}$ is a parabolic automorphism with fixed point $e^{i\theta}$, and $\psi$ is continuous on $\overline{\mathbb{D}}$ and bounded below. Then the spectrum of $W_{\psi,\varphi}$ is the circle centered at the origin of radius $\psi(e^{i\theta})$. 
        \item[4.] Suppose $\varphi$ is a hyperbolic automorphism with Denjoy-Wolff point $\xi$ and secondary fixed point $\xi'$. Let $\psi$ be continuous on $\overline{\mathbb{D}}$, $\psi'$  bounded on $\mathbb{D}$, and $\psi$ bounded below.  
        \begin{itemize}
            \item[a.] If $\psi(\xi)/\psi(\xi')>1$, then the point spectrum of $W_{\psi,\varphi}$ contains the two annuli
            \begin{align*}
                \{\lambda \in \mathbb{C}:|\psi(\xi)|\sqrt{\varphi'(\xi)}<|\lambda|<|\psi(\xi)|\sqrt{\left(\varphi'(\xi)\right)^{-1}}\}
            \end{align*}
            and
            \begin{align*}
                \{\lambda \in \mathbb{C}:|\psi(\xi')|\sqrt{\varphi'(\xi)}<|\lambda|<|\psi(\xi')|\sqrt{\left(\varphi'(\xi)\right)^{-1}}\}.
            \end{align*}
            \item[b.] If $1<\psi(\xi)/\psi(\xi')<1/\varphi'(\xi)$, then the spectrum of $W_{\psi,\varphi}$ is the annulus
            \begin{align*}
                 \{\lambda \in \mathbb{C}:|\psi(\xi')|\sqrt{\varphi'(\xi)}\leq|\lambda|\leq|\psi(\xi)|\sqrt{\left(\varphi'(\xi)\right)^{-1}}\}.
            \end{align*}
            \item[c.] If $\psi(\xi)/\psi(\xi')<\varphi'(\xi)$, then the point spectrum is empty.
        \end{itemize}
    \end{itemize}
\end{thm}

Finally, by combining the different characteristics of the classes of disc automorphisms, the spectral restriction on $W$ coming from Lemma \ref{spec_lem}, and information on the spectrum of WCO for each of these automorphism classes, we arrive at the main result in Theorem \ref{approx_dual} for this section. 

\begin{thm}\label{approx_dual}
    Assume $0\leq \varepsilon< 1$ and $\{f_{n}\}$ is a frame for $H^{2}(\mathbb{D})$ with bounds $0<A\leq B<\infty$. Let $W=W_{\psi,\varphi}:H^{2}(\mathbb{D})\to H^{2}(\mathbb{D})$ be a bounded WCO. 
    \begin{itemize}
        \item[1.] Suppose $\psi$ is continuous on $\overline{\mathbb{D}}$. Then $\{Wf_{n}\}$ is an $\varepsilon$-approximate dual frame to $\{f_{n}\}$ and $\varphi$ is elliptic if and only if $\varphi(z)=z$ and $\{M_{\psi}f_{n}\}$ is an $\varepsilon$-approximate dual frame to $\{f_{n}\}$, where $M_{\psi}$ is the multiplication operator by $\psi$.
        \item[2.] Suppose $\psi$ is continuous on $\overline{\mathbb{D}}$ and $\varphi$ is a parabolic automorphism. Then $\{Wf_{n}\}$ cannot be an $\varepsilon$-approximate dual to $\{f_{n}\}$. 
        \item[3.] Let $\varphi:\mathbb{D}\to\mathbb{D}$ be a hyperbolic automorphism with Denjoy-Wolff point $\xi$ and secondary fixed point $\xi'$. Suppose $\psi$ is continuous on $\overline{\mathbb{D}}$, $\psi'$ is bounded on $\mathbb{D}$, and $\psi$ is bounded below. If either $\psi(\xi)/\psi(\xi')>1$ or $1<\psi(\xi)/\psi(\xi')<1/\varphi'(\xi)$, then $\{Wf_{n}\}$ cannot be an $\varepsilon$-approximate dual frame to $\{f_{n}\}$.
    \end{itemize}
\end{thm}
\begin{proof}
    We begin with the first part. Note that the backward direction is trivial as $\varphi(z)=z$ gives $\{W_{\psi, \varphi} f_n\}= \{M_{\psi}f_n\}$ as an $\varepsilon$-approximate dual to $\{f_n\}$, and also follows from Theorem~\ref{spectrum_WCO}, with $\varphi(z)$ being the identity map and thus elliptic. 
    For the forward direction, suppose $\varphi(z)\neq z$. By Lemma \ref{spec_lem} we have that $W$ is invertible and $\sigma(W)\subset\{\lambda \in \mathbb{C}: \text{Re}\lambda\geq (C/B)\}$ for some constant $C>0$. However, under the elliptic and non-identity supposition, parts $1$ and $2$ of Theorem \ref{spectrum_WCO} imply that either 
    \begin{itemize}
        \item[(a)] the spectrum of $W$ is the closure of 
        \begin{align*}
            \{\mu:\mu^{n}=(\psi\circ\phi(\xi))(\psi\circ\phi(\lambda\xi))\cdots(\psi\circ\phi(\lambda^{n-1}\xi)),\xi\in\mathbb{D}\}
        \end{align*}
        where $\lambda=e^{2\pi i/n}$, $a$ is the fixed point of $\varphi$, and $\phi(z)=(a-z)/(1-\overline{a}z)$ or,
        \item[(b)] the spectrum of $W$ is a circle centered at the origin of radius $|\psi(\xi)|$, where $\xi$ is the fixed point of $\varphi$.
    \end{itemize}
    In either case, picking $\lambda$ with $\text{Re}\lambda=0$ produces a contradiction. Hence $\varphi(z)=z$ and the $\varepsilon$-approximate dual frame condition reduces to an $\varepsilon$-approximate dual frame condition on $M_{\psi}$. 
    
    For the second part note that if $\{Wf_{n}\}$ were an $\varepsilon$-approximate dual frame to $\{f_{n}\}$, Lemma \ref{spec_lem} implies that $\sigma(W)\subset\{\lambda \in \mathbb{C}: \text{Re}\lambda\geq (C/B)\}$. However, the third part of Theorem \ref{spectrum_WCO} under the supposition implies that $\sigma(W)$ is a circle centered at the origin of radius $|\psi(e^{i\theta})|$ where $e^{i\theta}$ is the fixed point of $\varphi$. Taking $\lambda\in\sigma(W)$ with $\text{Re}\lambda=0$ gives the contradiction. 

    Similarly for the third part, the final portion of Theorem \ref{spectrum_WCO} implies that either the point spectrum contains two annuli, or that the spectrum is an annulus itself. In either case, taking $\lambda$ with $\text{Re}\lambda=0$ gives a contradiction to Lemma \ref{spec_lem}. Thus, $\{Wf_{n}\}$ cannot be an $\varepsilon$-approximate dual frame to $\{f_{n}\}$. 
\end{proof}

The above result essentially reduces the hope of obtaining WCOs which produce approximate duals to considering only the multiplication operator, at least under mild conditions on the multiplication operator. In light of Theorem \ref{no_dual}, one would hope that producing an approximate dual via a WCO would imply that the underlying frame $\{f_{n}\}$ must be some weakened notion of a tight frame. The standard form of relaxation is through approximately tight frames. For convenience of readers, the definition of approximately tight frame is given here.
\begin{definition}
    Let $\mathcal{H}$ be a separable Hilbert space and $\{f_{n}\}$ be a frame for $\mathcal{H}$ with frame bounds $0<A\leq B<\infty$. Given $0\leq\delta<1$, we say that $\{f_{n}\}$ is a $\delta$-approximate tight frame if 
    \begin{align*}
        \left|\frac{B}{A}-1\right| \leq \delta. 
    \end{align*}
\end{definition}

However, the following proposition shows that this is not possible.

\begin{prop}\label{thm:hardy_approximate_dual}
    For all $\varepsilon, \delta \in (0,1)$, there exists a frame $\{f_{n}\}$ of $H^{2}(\mathbb{D})$ and a bounded weighted composition operator $W=W_{\psi,\varphi}:H^{2}(\mathbb{D})\to H^{2}(\mathbb{D})$ such that $\{Wf_{n}\}$ is an $\varepsilon$-approximate dual frame to $\{f_{n}\}$, but $\{f_{n}\}$ is not a $\delta$-approximate tight frame.
\end{prop}

\begin{proof}
Fix $\varepsilon, \delta \in (0,1)$ and $M > 1+\delta$. We construct a bounded, invertible
weighted composition operator $W$ on $H^2(\mathbb{D})$ and a frame
$\{f_n\}$ with optimal frame bounds $A = \tfrac{1}{M}$ and $B = 1$ such
that $\{W f_n\}$ is an $\varepsilon$-approximate dual of $\{f_n\}$. Since
$M$ is arbitrary, this rules out any form of approximate tightness.

First, we consider a simply connected region $R = \{ x+iy: 1<x<M, |y|<\varepsilon \} \subset \C$. Suppose $\psi: \D \to R$ such that $\psi(\D)=R$ (by the Riemann Mapping Theorem \cite{ahlfors1979complex}, $\psi$ is bijective, holomorphic and unique). It can be readily seen that 
\begin{enumerate}
  \item $1 < \text{Re}{\psi} < M$ and $|\text{Im}\psi| < \varepsilon$ on $\D$.
    \item $\inf_{\D} \text{Re} \psi =1$ and $\sup_{\D} \text{Re} \psi =M$.
\end{enumerate}
Moreover, since $R$ is simply connected and $\psi$ is bounded and holomorphic, $\psi$ has nontangential boundary limits to $\T$. Denote the extension of $\psi$ to $\T$ by $\psi^*$, then on $\T$, $1\leq \text{Re}(\psi^*) \leq M$ and $|\text{Im}(\psi^*)|\leq \varepsilon$ a.e. Now take $W = W_{\psi, \varphi}$ with the above $\psi$ and $\varphi(z) =z$. $W$ is a well-defined weighted composition operator because $||W|| = ||\psi||_{\infty} \leq \sqrt{M^2 + \varepsilon^2}$ (therefore $\psi, \tfrac{1}{\psi} \in H^{\infty}(\D)$). By definition, $W = M_{\psi}$, the multiplicative operator by $\psi$. Moreover, $W$ is invertible with $W^{-1} = M_{1/\psi}$. Decompose $W$ into $W = P +iQ$ where
$$
P : = \frac{1}{2}(W+W^*), \quad Q : = \frac{1}{2i}(W-W^*).
$$
Notice that $P$ and $Q$ are self-adjoint. Moreover, for all $ f \in H^2(\D)$,
$$\langle Pf,f\rangle = \text{Re}\langle \psi f, f\rangle=\int_\T \text{Re}(\psi^*) |f|^2 dm,$$
and
$$\langle Qf,f\rangle = \text{Im}\langle \psi f, f\rangle=\int_\T \text{Im}(\psi^*) |f|^2 dm.$$
\\
Hence the spectrum of $P$ satisfies $\sigma (P)\subset [1,M]$ and $Q$ satisfies $||Q||\leq \varepsilon$.  Let $k_z:=K_z/\|K_z\|$ be the normalized reproducing kernel at $z\in\D$.  Since
$\ip{M_{\psi}k_z}{k_z}=\psi(z)$, we have
\[
  \ip{Pk_z}{k_z}=\text{Re}\ip{M_{\psi}k_z}{k_z}=\text{Re}\psi(z),
  \qquad z\in\D ,
\]
so the numerical range of $P$ contains $\text{Re}\psi(\D)=(1,M)$. Therefore, 
$$\min(\sigma (P))=1, \max(\sigma (P))=M.$$

 Recall that if $V$ is a bounded and invertible operator and $\{ e_n \}$ is an orthonormal basis, then $\{ Ve_n \}$ forms a Riesz basis with frame operator  $VV^*$. Consider $e_n(z) = z^n$, an orthonormal basis of $H^2(\D)$ and $V= P^{-1/2}$. Then $\{f_n \}$, where $f_n = P^{-1/2}z^n$ forms a frame with optimal lower bound $\tfrac{1}{M}$ and optimal upper bound $1$. Moreover, its associated frame operator $S$ satisfies:
 \begin{align*}
     S=VV^{*}=P^{-1}.
 \end{align*}
 
Now we show $\{Wf_n\}$ is indeed an $\varepsilon$-approximate dual frame to $\{f_n\}$. Since $W = P+ iQ$ and $S= P^{-1}$,
\begin{align*}
    ||I -WS || &= || I - (P+iQ)P^{-1}||
    = ||-i QP^{-1}|| \leq \varepsilon,
\end{align*}
where the last inequality is due to $\min(\sigma (P))=1$, $\max(\sigma (P))=M$ and  $||Q|| \leq \varepsilon$. Since $M$ is arbitrarily large, $\{f_n\}$ cannot be a $\delta-$approximate tight frame for any $\delta \in (0,1)$. 
\end{proof}

\section{Duality Preservation on \texorpdfstring{$\mathcal{F}^{2}(\mathbb{C}^{n})$}{Fock space}}\label{sec:dualpreservation_Fock}

In the previous section, we observed that there are several rigidities in terms of allowable WCOs which produce approximately dual frames. However, as we will see, several of these rigidities remain while others are introduced for WCOs on the Bargmann-Fock space $\mathcal{F}^{2}(\mathbb{C}^{n})$. First, we require the following lemma. 
\begin{lem}\label{wco_fock_lem}
    Let $W = W_{\psi,\varphi}$ be a bounded, invertible weighted
composition operator on $\mathcal{F}^2(\mathbb{C}^n)$. Then $W$ is a scalar
multiple of a unitary operator. More precisely, for some $A$ unitary and $b \in \mathbb{C}^n$,
\[
\varphi(z) = Az + b, \qquad
\psi(z) = \psi(0)\, e^{-\langle z, A^*b\rangle},
\]
and $W$ is a constant multiple of an isometry: for every
$f \in \mathcal{F}^2(\mathbb{C}^n)$,
\[
\|Wf\| = r\|f\|,\; \text{where} \; r := |\psi(0)|\, e^{|b|^2/2}.
\]
\end{lem}
\begin{proof}
    By Theorem 1 of \cite{zhao2015invertible}, it is known that $W=W_{\psi,\varphi}$ is invertible if 
    \begin{align*}
        \varphi(z)=Az+b
    \end{align*}
    for some invertible matrix $A$ with $||A||=1$ and $b\in\mathbb{C}^{n}$, and there are constants $0<m\leq M<\infty$ such that
    \begin{align*}
        m\leq \left|\psi(z)\right|^{2}\exp\left(\frac{\left|\varphi(z)\right|^{2}-|z|^{2}}{2}\right)\leq M.
    \end{align*}
Since $||A||=1$, $|Az + b|^2 - |z|^2 = 2\text{Re}\langle Az, b\rangle
+ |b|^2$, so the entire function
$g(z) := \psi(z)e^{\langle Az, b\rangle}$ satisfies
$m e^{-|b|^2} \le |g(z)|^2 \le M e^{-|b|^2}$. Being bounded and
entire, $g$ is constant by Liouville's theorem; evaluating at $z = 0$
gives $g \equiv \psi(0)$, that is,
$\psi(z) = \psi(0)e^{-\langle Az, b\rangle}
         = \psi(0)e^{-\langle z, A^*b\rangle}$.
         
 Now we prove the norm identity of $W$. Let $w = Az + b$. Since the modulus is
invariant under the unitary $A$ and under translation, it holds that 
$\langle Az, b\rangle = \langle w, b\rangle - |b|^2$ and
$$|z|^2 = |w - b|^2 = |w|^2 - 2\text{Re}\langle w, b\rangle + |b|^2.$$
Therefore,
\[
-2\text{Re}\langle Az, b\rangle - |z|^2
= -2\text{Re}\langle w, b\rangle + 2|b|^2
  - |w|^2 + 2\text{Re}\langle w, b\rangle - |b|^2
= |b|^2 - |w|^2 .
\]
We conclude that
\[
\|Wf\|^2
= |\psi(0)|^2\pi^{-n}\int_{\mathbb{C}^n}
  e^{-2\text{Re}\langle Az, b\rangle}\, |f(Az+b)|^2\, e^{-|z|^2}\, dV(z)
= |\psi(0)|^2 e^{|b|^2} \|f\|^2.
\]
\end{proof}
Similarly to Theorem \ref{no_dual} and Remark~\ref{remark:canonical}, we again have that the only WCOs which produce exact dual frames are the multiplicative ones, and the associated dual frames are the canonical duals of given frames. 
\begin{prop}\label{fock_tight}
Let $\{f_n\}$ be a frame for $\mathcal{F}^2(\C^n)$ with frame operator $S$ and let $W$ be a bounded weighted composition operator on $\mathcal{F}^2(\C^n)$ such that $\{Wf_n\}$ is a dual frame of $\{f_n\}$.  Then $W=rI$ for some $r>0$ and $\{f_n\}$ is tight with bound $\frac{1}{r}$. Furthermore, $W=S^{-1}$, $\{ Wf_n\}$ is the canonical dual of $\{f_{n}\}$, and $\{ Wf_n\}$ is tight.
\end{prop}
\begin{proof}
Similar to Theorem~\ref{no_dual}, we can easily obtain $WS=I$; since $S$ is invertible
(Lemma \ref{lemma:frame_operator}), $W$ is invertible and $S=W^{-1}$.  By Lemma
\ref{wco_fock_lem} there is $r>0$ with $\|Wg\|=r\|g\|$ for every
$g\in \mathcal{F}^{2}(\C^{n})$.  Applying this to $g=Sf$ gives
$\|f\|=\|WSf\|=r\|Sf\|$, so $\|Sf\|=r^{-1}\|f\|$ for every $f$.  Since $S$ is
self-adjoint,
\[
  \ip{S^{2}f}{f}=\ip{Sf}{Sf}=\|Sf\|^{2}=r^{-2}\|f\|^{2}
  \qquad\text{for all }f\in \mathcal{F}^{2}(\C^{n}),
\]
so the self-adjoint operator $S^{2}-r^{-2}I$ has identically vanishing quadratic
form and therefore is trivial.  Thus $S^{2}=r^{-2}I$; as $S$ is positive and the
positive square root of $r^{-2}I$ is unique, $S=r^{-1}I$.  Hence
$W=S^{-1}=rI$, and
$\sum_n|\ip{f}{f_n}|^{2}=\ip{Sf}{f}=r^{-1}\|f\|^{2}$, so $\{f_n\}$ is tight with
frame bound $1/r$. Furthermore, $\{ Wf_n\}$ is the canonical dual frame of $\{f_{n}\}$, and thus $\{ Wf_n\}$ is also tight.
\end{proof}
Relaxing exact duality changes the picture completely in $H^2(\D)$, where Proposition~\ref{thm:hardy_approximate_dual} permits arbitrarily large frame upper bound, but not in $\mathcal{F}^2(\C^n)$. By Proposition~\ref{fock_main} the only WCO available is the scalar factor, so the question becomes how well a single scalar $c$ can approximate $S^{-1}$; taking $c=2/(A+B)$ gives
error $(B-A)/(A+B)$, and Proposition~\ref{fock_main} shows that this estimate is sharp.
\begin{prop}\label{fock_main}
Let $\{f_n\}$ be a frame for $\mathcal{F}^2(\mathbb{C}^n)$ with optimal frame
bounds $0 < A \le B < \infty$, and let $0 \le \varepsilon < 1$. Then
there exists a bounded weighted composition operator $W$ on
$\mathcal{F}^2(\mathbb{C}^n)$ such that $\{Wf_n\}$ is an
$\varepsilon$-approximate dual frame to $\{f_n\}$ if and only if
\[
\frac{B - A}{A + B} \le \varepsilon,
\qquad\text{equivalently,}\qquad
\frac{B}{A} \le \frac{1 + \varepsilon}{1 - \varepsilon}.
\]
\end{prop}
\begin{proof}
The two inequalities are equivalent since
$B(1-\varepsilon) \le A(1+\varepsilon)$ if and only if
$B - A \le \varepsilon(A + B)$.

($\Leftarrow$) Let $c = \tfrac{2}{A+B}$ and take $W = cI$, a weighted
composition operator with constant weight $\psi \equiv c$ and
$\varphi = \mathrm{Id}$, where $\mathrm{Id}$ is the identity map for $\mathbb{C}^n$. Let $S$ be the frame operator of $\{f_n\}$.  Since $S$ is self-adjoint with
$\sigma(S) \subseteq [A, B]$, it holds that
\[
\|I - WS\| = \|I - cS\|
= \max_{t \in \sigma(S)} |1 - ct|
\le \max_{t \in [A,B]} |1 - ct|
= \frac{B-A}{A+B} \le \varepsilon.
\]
Note
that $\{Wf_n\} = \{cf_n\}$ is a frame trivially. Therefore, $\{Wf_n\}$ is an $\varepsilon$-approximate dual frame to $\{f_n\}$.

($\Rightarrow$) Suppose $\|I - WS\| \le \varepsilon < 1$. Then $WS$
is invertible by a Neumann series argument, and since $S$ is
invertible, so is $W$. By Lemma \ref{wco_fock_lem} there is
$r > 0$ with $\|Wg\| = r\|g\|$ for all $g \in \mathcal{F}^2(\mathbb{C}^n)$.
Then for every unit vector $f$,
\[
\bigl|\,1 - r\|Sf\|\,\bigr| = \bigl|\,\|f\| - \|WSf\|\,\bigr|
   \le \|(I-WS)f\| \le \varepsilon,
\]
and
\begin{equation}\label{eq:4.3ineq}
1 - \varepsilon \le r\|Sf\| \le 1 + \varepsilon,
\qquad \|f\| = 1.
\end{equation}
Since the frame bounds are optimal and $S$ is positive and self-adjoint,
$\|Sf\|^2 = \langle S^2 f, f\rangle$ which indicates that
\[
\inf_{\|f\|=1} \|Sf\| = A
\qquad\text{and}\qquad
\sup_{\|f\|=1} \|Sf\| = B
\]
both attained in the limit along the approximate eigenvectors of $S$ at
$A$ and $B$ respectively. Taking such approximate eigenvectors in
\eqref{eq:4.3ineq} yields
\[
rA \ge 1 - \varepsilon
\qquad\text{and}\qquad
rB \le 1 + \varepsilon,
\]
and dividing gives $B/A \le (1+\varepsilon)/(1-\varepsilon)$.
\end{proof}
\begin{remark}
Proposition \ref{fock_main} recovers Proposition \ref{fock_tight} as the case $A = B$, i.e.\ tightness. Notice the sharp contrast to Proposition~\ref{thm:hardy_approximate_dual}: on $H^2(\mathbb{D})$,
weighted composition operators produce $\varepsilon$-approximate
duals of frames with arbitrarily large condition number $B/A$,
whereas on $\mathcal{F}^2(\mathbb{C}^n)$, the approximate dual via a weighted
composition operator is possible precisely for approximately tight
frames, with the sharp bounds satisfying
$B/A \le (1+\varepsilon)/(1-\varepsilon)$. 
\end{remark}
\section*{Acknowledgment}
The authors would like to thank Professor Ruhan Zhao for helpful discussions. 

\section*{Competing Interests}
The authors declare no competing interests.
\bibliographystyle{plain}
\bibliography{refs}
\end{document}